\documentclass{amsart}
\usepackage{amsmath,amsfonts,amsthm}
\usepackage{amssymb,latexsym}
\usepackage{graphics}
\usepackage[colorlinks]{hyperref}

\usepackage{amssymb}

\theoremstyle{plain}
\theoremstyle{definition}
\newtheorem{theorem}{Theorem}
\newtheorem{lemma}[theorem]{Lemma}

\newtheorem{definition}[theorem]{Definition}

\newtheorem{remark}[theorem]{Remark}
\theoremstyle{remark}

\numberwithin{equation}{section}
\newcommand{\SP}{\: \: \: \: \:}

\title[Compact non--uniformizable Li--Yorke chaotic dynamical systems]{Compact non--uniformizable Li--Yorke chaotic dynamical systems via an example}
\author[M. Pourattar, F. Ayatollah Zadeh Shirazi]{Mehrnaz Pourattar, Fatemah Ayatollah Zadeh Shirazi}
\begin{document}
\begin{abstract}
The main aim of this paper is extending the concept of scambled pair and
Li--Yorke chaos to non--uniform compact dynamical systems. 
We show for finite (compact Alexandroff) topological space
$X$ with at least two elements the following statements are equivalent:
\begin{itemize}
\item one--sided shift $\sigma:X^{\mathbb{N}}\to X^\mathbb{N}$ is  Li--Yorke chaotic,
\item one--sided shift $\sigma:X^{\mathbb{N}}\to X^\mathbb{N}$ has at least one 
scrambled pair,
\item one--sided shift $\sigma:X^{\mathbb{N}}\to X^\mathbb{N}$ has at least one 
non--asymptotic pair,
\item there exists $a,b\in X$ such that $\overline{\{a\}}\cap\overline{\{b\}}=\varnothing$,
\item $\bigcap\{\overline{\{a\}}:a\in X\}=\varnothing$.
\end{itemize}
\end{abstract}
\maketitle
\noindent {\small {\bf 2020 Mathematics Subject Classification:}  37B02, 54C05   \\
{\bf Keywords:}}  Alexandroff space. asymptotic pair, Li--Yorke chaos, one--sided shift, scrambled pair.
\section{Introduction}
\noindent By dynamical system $(X,f)$ or less formally $f:X\to X$ we mean a topological
space $X$ and continuous map $f:X\to X$. 
The idea of Li and Yorke in \cite{li} have been followed by so many mathematicians
since 1975. Li--Yorke chaos and its related topics 
has been considered in unit interval (e.g., \cite{butler, du1, du2, mizera}), compact metric (e.g. \cite{lou}) and uniform (e.g. \cite{rezavand}) dynamical systems, 
our main aim is to generalize 
the concept in general compact topological spaces.
\section{Preliminaries}
\noindent Let's make a glance on Li--Yorke chaotic compact metric and compact uniform
dynamical systems.
\begin{remark}\label{metric}
In compact metric space $(X,d)$ and continuous map $f:X\to X$ 
we say $(x,y)\in X\times X$ is an scrambled pair (or $x,y$ are scrambled) if the following
two conditions hold:
\begin{itemize}
\item[(m-1)] $\mathop{\liminf}\limits_{n\to\infty}d(f^n(x),f^n(y))=0$,
\item[(m-2)] $\mathop{\limsup}\limits_{n\to\infty}d(f^n(x),f^n(y))>0$.
\end{itemize}
We also say $A\subseteq X$ is an scrambled subset of $X$ if each distinct 
points $x,y\in A$ are scambled. $f:X\to X$ is
Li--Yorke chaotic if $X$ has an uncountable scrambled set (see e.g. \cite{blan}).
\end{remark}
\noindent In compact metric space $X$, $\{U\subseteq X\times X:\exists\varepsilon>0\:\{(x,y)\in X\times X:d(x,y)<\varepsilon\}\subseteq U\}$ is
the unique compatible uniformity on $X$.
\\
For arbitrary set $A$ let $\Delta_A:=\{(a,a):a\in A\}$. Moreover
let's mention that in compact Hausdorff space $X$, $\{U\subseteq X\times X:\Delta_X
\subseteq U^\circ\}$ is the unique compatible uniformity on $X$. Let's go ahead  to Li--Yorke chaos in compact Hausdorff (hence uniform) dynamical systems. 
\begin{remark}\label{uniform}
In compact Hausdorff uniform space $(X,\mathcal{U})$ and continuous map $f:X\to X$ we say $(x,y)\in X\times X$ is an scrambled pair (or $x,y$ are scrambled) if
the following
two conditions hold:
\begin{itemize}
\item[(u-1)] for each $U\in\mathcal{U}$, $\{n\in\mathbb{N}:(f^n(x),f^n(y))\in U\}\neq\varnothing$,
\item[(u-2)] there exists $U\in\mathcal{U}$ such that $\{n\in\mathbb{N}:(f^n(x),f^n(y))\notin U\}$ is infinite.
\end{itemize}
We also say $A\subseteq X$ is an scrambled subset of $X$ if each distinct points $x,y\in A$ are scambled. $f:X\to X$ is
Li--Yorke chaotic if $X$ has an uncountable scrambled set (see e.g. \cite{rezavand, attar}).
\end{remark}
\noindent Now we are ready to have a definition of Li--Yorke chaos in general compact dynamical systems. 
\begin{definition}\label{general}
In compact  space $X$ and continuous map $f:X\to X$ we say $(x,y)\in X\times X$ is an scrambled pair (or $x,y$ are scrambled) if the following two conditions hold:
\begin{itemize}
\item[(g-1)] for each open neighbourhood $U$ of $\Delta_X$ in $X\times X$,
	$\{n\in\mathbb{N}:(f^n(x),f^n(y))\in U\}\neq\varnothing$,
\item[(g-2)] there exists open neighbourhood $U$ of $\Delta_X$ in $X\times X$ 
	such that $\{n\in\mathbb{N}:(f^n(x),f^n(y))\notin U\}$ is infinite.
\end{itemize}
We also say $A\subseteq X$ is an scrambled subset of $X$ if each distinct points $x,y\in A$ are scrambled. $f:X\to X$ is
Li--Yorke chaotic if $X$ has an uncountable scrambled set.
\end{definition}
\begin{definition}
In dynamical system $f:X\to X$ we say  $(x,y)\in X\times X$ is a proximal pair (or $x,y$ are proximal) if there exists $z\in X$
and a net $\{n_\alpha\}_{\alpha\in\Lambda}$ in $\mathbb N$ such that both nets $\{f^{n_\alpha}(x)\}_{\alpha\in\Lambda}$ and 
$\{f^{n_\alpha}(y)\}_{\alpha\in\Lambda}$ converge to $z$ \cite{ellis}. Note that in dynamical system $f:X\to X$:
\\
$\bullet$ if $X$ is compact metric with compatible metric $d$, $x,y\in X$ are proximal if and only if  (m-1) holds.
 \\
 $\bullet$ if $X$ is compact Hausdorff with compatible uniformity $\mathcal U$, $x,y\in X$ are proximal if and only if  (u-1) holds.
 \\
  $\bullet$ if $X$ is compact, $x,y\in X$ are proximal if and only if  (g-1) holds.
\end{definition}
\begin{definition}
In dynamical system $f:X\to X$ we say  $(x,y)\in X\times X$ is an asymptotic pair (or $x,y$ are asymptotic) if 
$\{n\in\mathbb{N}:(f^n(x),f^n(y))\notin U\}$ is finite for each open neighbourhood $U$ of $\Delta_X$ in $X\times X$.
Note that in dynamical system $f:X\to X$:
\\
$\bullet$ if $X$ is compact metric with compatible metric $d$, $x,y\in X$ are asymptotic if and only if (m-2) does not hold.
 \\
$\bullet$ if $X$ is compact Hausdorff with compatible uniformity $\mathcal U$, $x,y\in X$ are asymptotic if and only if  (u-2) does not hold.
\\
Hence:
\\
$\bullet$ compact metric dynamical system $f:X\to X$ satisfies Definition~\ref{metric} if and only if it satisfies Definition~\ref{uniform} (resp.
	Definition~\ref{general}),
\\
$\bullet$ compact Hausdorff dynamical system $f:X\to X$ satisfies Definition~\ref{uniform} if and only if it satisfies Definition~\ref{general}.
\\
$\bullet$ In compact dynamical system $f:X\to X$, points $x,y\in X$ are scrambled if and only
if they are proximal and non-asymptotic.
\end{definition}
\section{Exploring via an example}
\noindent In this section in topological space $X$, equip $X^\mathbb{N}$
with product (pointwise convergence) topology and consider one--sided shift  $\mathop{\sigma:X^\mathbb{N}\to X^\mathbb{N}\SP\SP}\limits_{
(x_n)_{n\in\mathbb N}\mapsto(x_{n+1})_{n\in\mathbb N}}$. 
\\
Alexandroff spaces have been introduced by P. Alexandroff in \cite{alex}.
Various sub--categories of  topological spaces devoted to Alexandroff spaces, like
finite topological spaces \cite{finite}, functional Alexandroff spaces \cite{func},
Khalimsky spaces \cite{khal}, etc..
\\
Let's recall that a topological space $X$ is an Alexandroff  space if intersection of any nonempty 
collection of  $X$ is open (or equivalently each point has a smallest open neighbourhood).
\\
If $X$ is an Alexandroff space with at least two elements,
for each $x\in X$ suppose
$V(x)$ denotes the smallest open neighbourhood of $x$. 
For $x_1,\ldots, x_n\in X$ let $G(x_1,\ldots,x_n)=V(x_1)\times V(x_2)\times\cdots\times V(x_n)\times X\times X\times\cdots$, then 
$\{G(x_1,\ldots,x_n):n\in\mathbb{N}$
and $x_1,\ldots, x_n\in X\}$ is a topological basis of $X^\mathbb{N}$. For
$n\in\mathbb{N}$ also let $\Gamma_n=\bigcup\{G(x_1,\ldots,x_n)\times G(x_1,\ldots,x_n):x_1,\ldots, x_n\in X\}$.
\begin{lemma}\label{salam10}
Suppose $X$ is a compact Alexandroff space and $U\subseteq X^\mathbb{N}\times X^\mathbb{N}$.
There exists $N\in\mathbb{N}$ with $\Gamma_N\subseteq U$ if and only if $\Delta_{X^\mathbb{N}}$ contained in the interior of $U$.
\end{lemma}
\begin{proof}
Consider $N\in\mathbb{N}$, it is evident that 
$\Gamma_N$ is an open subset of $X^\mathbb{N}\times X^\mathbb{N}$ moreover for
$a=(a_n)_{n\in\mathbb{N}}\in X^\mathbb{N}$ we have $(a,a)\in G(a_1,\ldots,a_N)\times G(a_1,\ldots,a_N)\subseteq \Gamma_N$, thus $\Delta_{X^\mathbb{N}}\subseteq \Gamma_N$. Hence
$\Gamma_N$ is an open subset of $X^\mathbb{N}\times X^\mathbb{N}$ 
containing $\Delta_{X^\mathbb{N}}$.
\\
Now suppose $V$ is an open subset of $X^\mathbb{N}\times X^\mathbb{N}$ containing 
$\Delta_{X^\mathbb{N}}$. By compactness of $X$ and Tychonoff's theorem, $X^{\mathbb N}$
is compact. The continuity of $\mathop{X^{\mathbb N}\to X^{\mathbb N}\times X^{\mathbb N}}\limits_{x\mapsto(x,x)\SP\SP}$ leads to the compactness of $\Delta_{X^{\mathbb N}}$. For
each $z=(z_n)_{n\geq1}$ there exist $n_z\in\mathbb{N}$
and open subsets $U_1,\ldots,U_{n_z},W_1,\ldots,W_{n_z}$ of $X$ such that
$(z,z)\in (U_1\times\cdots\times U_{n_z}\times X\times X\times\cdots)\times 
(W_1\times\cdots\times W_{n_z}\times X\times X\times\cdots)\subseteq V$ thus
{\small\begin{eqnarray*}
(z,z)\in A^z&:=&G(z_1,\ldots,z_{n_z})\times G(z_1,\ldots,z_{n_z}) \\
&\subseteq &(U_1\times\cdots\times U_{n_z}\times X\times X\times\cdots)\times 
(W_1\times\cdots\times W_{n_z}\times X\times X\times\cdots)\\
&\subseteq& V\:.
\end{eqnarray*}}
Since $\Delta_{X^{\mathbb N}}\subseteq \bigcup \{A^z:z\in X^{\mathbb{N}}\}$ and
$\Delta_{X^{\mathbb N}}$ is compact there exist $z^1,\ldots, z^p\in X^{\mathbb N}$ such that
$\Delta_{X^{\mathbb N}}\subseteq A^{z^1}\cup\cdots\cup A^{z^p}$. Let 
\[N=\mathop{\max}\limits_{1\leq i\leq p}n_{z^i}\:.\]
We claim $\Gamma_N\subseteq V$. For each $y=(y_n)_{n\in\mathbb{N}}\in X^{\mathbb N}$
there exists $j\in\{1,\ldots,p\}$ such that $(y,y)\in A^{z^j}$ therefore $y_i\in V(z^j_i)$
and $V(y_i)\subseteq V(z^j_i)$
for each $i\in\{1,\ldots,n_{z^j}\}$ (where $z^j=(z^j_i)_{i\in\mathbb{N}}$) hence
$G(y_1,\ldots,y_{n_{z^j}})\subseteq G(z^j_1,\ldots,z^j_{n_{z^j}})$, so
\begin{eqnarray*}
G(y_1,\ldots,y_N)\times G(y_1,\ldots,y_N) &\subseteq &
G(y_1,\ldots,y_{n_{z^j}})\times G(y_1,\ldots,y_{n_{z^j}}) \\
&\subseteq &
G(z^j_1,\ldots,z^j_{n_{z^j}})\times G(z^j_1,\ldots,z^j_{n_{z^j}})\\
&=&A^{z^j} \subseteq U
\end{eqnarray*}
which leads to $\Gamma_N\subseteq U$.
\end{proof}
\begin{lemma}\label{salam20}
In Alexandroff space $X$ for nonempty subset $D$ of $X$ we have
\[\bigcap\{\overline{\{x\}}:x\in D\} =\{z\in X:D\subseteq V(z)\}\:.\]
In particular, $\bigcap\{\overline{\{x\}}:x\in D\} =\{z\in X:V(z)=X\}$.
\end{lemma}
\begin{proof}
Note that for each $E\subseteq X$ and $x\in X$, $x$ belongs to $\overline{E}$ if and only if
each open neighbourhood of $x$ intersects $E$ or equivalently $V(x)\cap E\neq\varnothing$.
Consider the following equations
\begin{eqnarray*}
\bigcap\{\overline{\{x\}}:x\in D\} & = & \{z\in X:\forall x\in D\: z\in \overline{\{x\}}\} \\
& = & \{z\in X:\forall x\in D\: V(z)\cap \{x\}\neq\varnothing\} \\
& = & \{z\in X:\forall x\in D\: x\in V(z)\} =\{z\in X:D\subseteq V(z)\}
\end{eqnarray*}
\end{proof}
\begin{lemma}\label{salam30}
In nonempty compact space $X$, the following statements are equivalent:
\begin{itemize}
\item[a.] $\bigcap\{\overline{\{x\}}:x\in X\}\neq\varnothing$,
\item[b.] for all $x,y\in X$ we have $\overline{\{x\}}\cap \overline{\{y\}}\neq\varnothing$,
\item[c.] there exists $z\in X$ such that $X$ is the unique open neighbourhood of $z$.
\end{itemize}
\end{lemma}
\begin{proof}
It is evident that (a) implies (b).
\\
Now suppose for all $x,y\in X$, $\overline{\{x\}}\cap \overline{\{y\}}\neq\varnothing$. We claim that for all $n\in\mathbb{N}$ and $y_1,\ldots,y_n\in X$, 
$\overline{\{y_1\}}\cap\cdots\cap \overline{\{y_n\}}\neq\varnothing$. For this aim note that
\begin{itemize}
\item for all $y_1,y_2\in X$ we have $\overline{\{y_1\}}\cap \overline{\{y_2\}}\neq\varnothing$ by the assumption.
\item Consider $k\geq2$ such that $\overline{\{y_1\}}\cap\cdots\cap \overline{\{y_k\}}\neq\varnothing$ for all $y_1,\ldots,y_k\in X$ also consider $y_{k+1}\in X$.
Choose $b\in \overline{\{y_1\}}\cap\cdots\cap \overline{\{y_k\}}$, then $\overline{\{b\}}
\subseteq \overline{\{y_1\}}\cap\cdots\cap \overline{\{y_k\}}$ and by the assumption 
$\overline{\{y_1\}}\cap\cdots\cap \overline{\{y_{k+1}\}}\supseteq \overline{\{b\}}\cap \overline{\{y_{k+1}\}}\neq\varnothing$ which shows $\overline{\{y_1\}}\cap\cdots\cap \overline{\{y_{k+1}\}}\neq\varnothing$.
\end{itemize}
Hence $\bigcap\{\overline{\{x\}}:x\in X\}$ is a nonempty collection of closed subsets of
compact space $X$ with finite intersection property, so $\bigcap\{\overline{\{x\}}:x\in X\}\neq\varnothing$. Hence (b) implies (a).
\\
In order to complete the proof note that $\bigcap\{\overline{\{x\}}:x\in X\}=\{x\in X:x$
is the unique open neighbourhood of $x\}$.
\end{proof}
\begin{remark}\label{salam40}
If $X=\{a,b\}$ is a discrete space with two elements, then 
$\sigma:\{a,b\}^{\mathbb N}\to\{a,b\}^{\mathbb N}$
is  Li--Yorke chaotic~\cite{rezavand}.
\end{remark}
\begin{theorem}\label{salam50}
In compact topological space $X$, if there exist $a,b\in X$ such that $\overline{\{a\}}\cap \overline{\{b\}}
=\varnothing$, then $\sigma:X^\mathbb{N}\to X^\mathbb{N}$ is 
Li--Yorke chaotic.
\end{theorem}
\begin{proof}
Choose $a,b\in X$ such that $\overline{\{a\}}\cap \overline{\{b\}}
=\varnothing$, then $\{a,b\}$with induced topology is a discrete space with two elements.
By Remark~\ref{salam40}, 
$\sigma\restriction_{\{a,b\}^{\mathbb N}}:\{a,b\}^{\mathbb N}\to\{a,b\}^{\mathbb N}$
is topological Li--Yorke chaotic. So $\sigma\restriction_{\{a,b\}^{\mathbb N}}:\{a,b\}^{\mathbb N}\to\{a,b\}^{\mathbb N}$ has an uncountable topological Li-Yorke scrambled set
like $A$. For all distinct $z,w\in A$:
\begin{itemize}
\item $z,w$ are proximal in 
	$\sigma\restriction_{\{a,b\}^{\mathbb N}}:\{a,b\}^{\mathbb N}\to\{a,b\}^{\mathbb N}$, 
	hence there exists a net $\{n_\alpha\}_{\alpha\in\Lambda}$ in
	$\mathbb N$ and $u\in \{a,b\}^{\mathbb N}$ such that nets 
	$\{\sigma^{n_\alpha}(z)\}_{\alpha\in\Lambda}$ and 
	$\{\sigma^{n_\alpha}(w)\}_{\alpha\in\Lambda}$ converge to $u$ in $\{a,b\}^{\mathbb N}$
	thus $\{\sigma^{n_\alpha}(z)\}_{\alpha\in\Lambda}$ and 
	$\{\sigma^{n_\alpha}(w)\}_{\alpha\in\Lambda}$ converge to $u$ in $X^{\mathbb N}$
	(note that $\{a,b\}^\mathbb{N}$ with product topology carries the same subspace
	topology of $X^\mathbb{N}$). Thus $z,w$ are proximal for 
	$\sigma:X^\mathbb{N}\to X^\mathbb{N}$.
\item $z,w$ are not  asymptotic in 
	$\sigma\restriction_{\{a,b\}^{\mathbb N}}:\{a,b\}^{\mathbb N}\to\{a,b\}^{\mathbb N}$, 
	so there exists open subset $U$ of $\{a,b\}^{\mathbb N}\times \{a,b\}^{\mathbb N}$
	containing $\Delta_{\{a,b\}^{\mathbb N}}$ such that 
	$\{n\in\mathbb{N}:(\sigma^n(z),\sigma^n(w))\notin U\}$ is infinite. 
	There exists open subset $W$ of  $X^{\mathbb N}$ containing
	$\Delta_{X^{\mathbb N}}$ such that $W\cap \{a,b\}^{\mathbb N}=U$.
	For each $n\in\mathbb{N}$, $\sigma^n(z),\sigma^n(w)\in\{a,b\}^\mathbb{N}$, thus
	$\{n\in\mathbb{N}:(\sigma^n(z),\sigma^n(w))\notin T\}=
	\{n\in\mathbb{N}:(\sigma^n(z),\sigma^n(w))\notin U\}$ is infinite too and 
	$z,w$ are not  asymptotic in $\sigma:X^\mathbb{N}\to X^\mathbb{N}$.
\end{itemize}
Hence $A$ is a  Li-Yorke scrambled subset of $X^\mathbb{N}$ in
dynamical system $\sigma:X^\mathbb{N}\to X^\mathbb{N}$, in particular
$\sigma:X^\mathbb{N}\to X^\mathbb{N}$ is   Li--Yorke chaotic.
\end{proof}
\begin{theorem}\label{salam60}
For compact Alexandroff space $X$ with at least two elements and 
one--sided shift $\sigma:X^\mathbb{N}\to X^\mathbb{N}$ the following statements are
equivalent:
\begin{itemize}
\item[a.] $\sigma:X^\mathbb{N}\to X^\mathbb{N}$ is 
	Li--Yorke chaotic,
\item[b.] $\sigma:X^\mathbb{N}\to X^\mathbb{N}$ 
	has at least two  Li--Yorke scrambled points,
\item[c.] $\sigma:X^\mathbb{N}\to X^\mathbb{N}$ 
	has at least two  non--asymptotic points,
\item[d.] there exist $a,b\in X$ such that 
	$\overline{\{a\}}\cap \overline{\{b\}}=\varnothing$, 
\item[e.]  $\bigcap\{\overline{\{x\}}:x\in X\}=\varnothing$,
\item[f.] For all $z\in X$, $V(z)\neq X$.
\end{itemize}
\end{theorem}
\begin{proof}
(d, e, f) are equivalent by Lemmas~\ref{salam20} and \ref{salam30}. By
Theorem~\ref{salam50}, (c) implies (a). Obviously (a) implies (b), also (b) implies (c). 
In order to show
(c) implies (f), suppose there exists $z\in X$ such that $V(z)=X$, then $X^\mathbb{N}
\times X^\mathbb{N}$ is the only open neighbourhood of $((z)_{n\in\mathbb N},(z)_{
n\in\mathbb N})$ in $X^\mathbb{N}\times X^\mathbb{N}$. So only open subset of 
$X^\mathbb{N}\times X^\mathbb{N}$ containing $\Delta_{X^\mathbb{N}}$ is
$X^\mathbb{N}\times X^\mathbb{N}$ itself, so for all $z,y\in X^\mathbb{N}$
and $n\in \mathbb{N}$ and open set $U$ containing 
$\Delta_{X^\mathbb{N}}$ we have $(\sigma^n(z),\sigma^n(y))\in U$, in particular 
$z,y$ are asymptotic.
\end{proof}
\noindent We say $\alpha\in A$ is a quasi--periodic point of self--map $f:A\to A$
if $\{f^n(\alpha):n\geq1\}$ is finite (or equivalently there exist $s>t\geq1$ such that
$f^s(\alpha)=f^t(\alpha)$).
\\
In the following remark, we pat attention to generalized shifts.
Suppose $\Lambda$ is a nonempty set and self--map $\varphi:\Lambda\to\Lambda$,
we call $\mathop{\sigma_\varphi:X^\Lambda\to X^\Lambda}\limits_{(x_\alpha)_{\alpha\in\Lambda}
\mapsto(x_{\varphi(\alpha)})_{\alpha\in\Lambda}}$ a generalized shift.
Generalized shift has been introduced for the first time in \cite{first}, however
dynamical and non--dynamical properties have been studied in several texts
(see e.g. \cite{rezavand, anna}).
\begin{remark}\label{salam70}
Consider arbitrary self--map $\varphi:\Lambda\to\Lambda$ and generalized shift 
\linebreak
$\sigma_\varphi:X^\Lambda\to X^\Lambda$, then:
\begin{itemize}
\item For finite discrete $X=\{a,b\}$ with two elements, by \cite{rezavand}
	$\sigma_\varphi:X^\Lambda\to X^\Lambda$ is (uniform) Li--Yorke chaotic if and only if
	$\varphi:\Lambda\to\Lambda$ has at least one non--quasi periodic point,
\item using a similar method described for Theorem~\ref{salam60}, for compact Alexandroff
	space $X$ the following statements are equivalent:
	\begin{itemize}
	\item $\sigma_\varphi:X^\Lambda\to X^\Lambda$ is  Li--Yorke chaotic,
	\item $\sigma_\varphi:X^\Lambda\to X^\Lambda$ 
		has at least two  Li--Yorke scrambled points,
	\item $\sigma_\varphi:X^\Lambda\to X^\Lambda$ 
		has at least two  non--asymptotic points,
	\item $\varphi:\Lambda\to\Lambda$ has at least one non--quasi periodic point and 
		there exist $a,b\in X$ such that 	$\overline{\{a\}}\cap \overline{\{b\}}=\varnothing$.
	\end{itemize}
\end{itemize}
\end{remark}

\noindent {\small {\bf Mehrnaz Pourattar}, 
Department of Mathematics, Science and Research Branch, Islamic
Azad University, Tehran, Iran
 (mpourattar@yahoo.com)}
\\
{\small {\bf Fatemah Ayatollah Zadeh Shirazi}, Faculty
of Mathematics, Statistics and Computer Science, College of
Science, University of Tehran, Enghelab Ave., Tehran, Iran
\linebreak (f.a.z.shirazi@ut.ac.ir)}
%
%
%
%
%
%
%
%
%

\end{document}